\documentclass[12pt,preprint]{article}
\usepackage[pdftex]{hyperref}
\usepackage{amstext}
\usepackage{latexsym}
\usepackage{amssymb}
\usepackage{amsmath}
\usepackage{amsthm}

\usepackage{color}
\usepackage{changes}
\def\beq{\begin{equation}}
	\def\eeq{\end{equation}}
\def\beqn{\begin{equation*}}
	\def\eeqn{\end{equation*}}
\def\baq{\begin{eqnarray}}
	\def\eaq{\end{eqnarray}}
\def\baqn{\begin{eqnarray*}}
	\def\eaqn{\end{eqnarray*}}

\newcommand{\Multi}{\,\,\lower 1pt
	\hbox{$\overrightarrow{\longrightarrow}$}\,\,}
\newcommand{\multi}{\,\,\lower 1pt
	\hbox{$\overrightarrow{\rightarrow}$}\,\,}
\def\image #1 (#2,#3) (echelle #4) #5{
	\dimen2=#2
	\dimen3=#3
	\divide \dimen2 by 1000
	\multiply \dimen2 by #4
	\divide \dimen3 by 1000
	\multiply \dimen3 by #4
	\setbox1 =\vbox to \dimen2{\hsize=\dimen3\vfill\special{picture #1
			scaled #4}}
	\vbox{\hsize=\dimen3\box1\medskip\centerline{#5}}
}
\newtheorem{definition}{Definition}[section]
\newtheorem{proposition}{Proposition}[section]
\newtheorem{theorem}{Theorem}[section]
\newtheorem{corollary}{Corollary}[section]
\newtheorem{lemma}{Lemma}[section]
\newtheorem{remark}{Remark}[section]

\newcommand{\ope}{\operatorname}

\numberwithin{equation}{section}
\newcommand{\Lim}{\ope{Lim}}

\begin{document}
	\title{Representation formulas for maximal monotone operators of type (D) in Banach spaces whose dual spaces are strictly convex}
	\author{ Nguyen B. Tran\thanks{nguyenbaotran31@gmail.com}, Tran N. Nguyen, Huynh M. Hien \\
		Department of Mathematics and Statistics,\\
		Quy Nhon University, \\
		170 An Duong Vuong, Quy Nhon, Vietnam \\
	}
	\date{}
	\maketitle
	
	\begin{abstract}

		This work deals with a maximal monotone operator $A$ of type (D) in a Banach space whose dual space is strictly convex. We establish some representations for the value $Ax$ at a given point $x$ via its values at nearby points of $x$. We show that the faces of $Ax$ are contained in the set of all  weak$^*$ convergent limits of bounded nets of the operator at nearby points of $x$, then we obtain a representation for $Ax$ by use of this set. In addition,  representations for the support function of $Ax$ based on the  minimal-norm selection of the operator in certain Banach spaces are given.

		
	\end{abstract}
	
	
	\section{Introduction}  
	The theory of maximal monotone operators has had a great development for over sixty years. The concept of maximal monotone operators is a generalization of  subdifferentials of lower semi-continuous convex functions. Maximal monotone operators rapidly found uses for subgradients, optimization, algorithms, financial mathematics, and much more.  The closedness of graphs of maximal monotone operators plays a significant role in the theory of variational analysis.  It is well-known that the graph of a maximal monotone operator is norm$\times$weak$^*$ sequentially closed. However,  the graph of such an operator might be not norm$\times$weak$^*$ closed.  This occurs due to the fact that a weak$^*$ convergent net might be unbounded.

	In 2003, Borwein et al.~\cite{BFG} gave an example showing that the subdifferential of a lower semi-continuous function in a separable Hilbert space is not norm$\times$weak$^*$ closed. 
	In 2012, Borwein and Yao~\cite{BY1} provided an explicit structure formula for maximal monotone operators whose domains have nonempty interiors. Consequently, the graphs of these operators are
	norm$\times$weak$^*$ closed. The authors considered a maximal monotone operator in a subset of  the interior of its domain in which the operator is locally bounded. The support function of the operator at a given point $x$ equals the support function in the set consisting of the weak$^*$ convergent limits of the operator at nearby points of $x$.   Recently, Hantoute and Nguyen~\cite{HB} considered maximal monotone operators in Hilbert spaces whose domains may have empty interiors
	and established a relation  between the faces of its values at a given point $x$ and  the set of strong convergent limits of the operator at nearby points of $x$. 
	Then, the authors provided  an explicit representation for maximal monotone operators.
	These results were extended to maximal monotone operators in reflexive Banach spaces 
	which are locally bounded together with their dual spaces by Khanh and Nguyen~\cite{NK}. Furthermore, based on  the minimal-norm selections,
	the authors discovered another formula for maximal monotone operators in that class of spaces. 
	
	The main approach in \cite{HB,NK} is that the authors used Minty surjectivity theorem in reflective spaces. A class of maximal monotone operators in nonreflexive Banach spaces which share several properties with these operators in reflexive Banach spaces was introduced by Gossez  in 1971 and named `type (D)'; see \cite{Go1}. A typical property of such a  maximal monotone operator is that the range of the sum of it and the duality mapping is dense in the dual space. In 1992, Fitzpatrick and Phelps~\cite{FP} defined a  class of maximal monotone operators which includes the class of type (D) and was later called  `type (FP)'. Four years later, Simons~\cite{Si1} introduced the class (NI) which was seemed to generalize the class of type (D). However, in 2010, Marques and Svaiter~\cite{MaSv1} showed that these two classes are identical. One year later, Bauschke et al.~\cite{BBWY} discovered that these three of types coincide.

	In this paper we consider a maximal monotone operator of  type (D) in a nonreflexive Banach space whose dual space is strictly convex and try to extend the results in \cite{HB,NK}. We represent the value of the operator at a point $x$ via the set 
	consisting of weak$^*$ convergent limits of its bounded nets at nearby points of $x$. We introduce the concepts `w-Kadec-Klee property' and `w$^*$-Kadec-Klee property' for Banach spaces and their dual spaces, respectively. 
	These classes of spaces include locally uniformly convex spaces.  Respective results are given for the case that $X$ has the w-Kadec-Klee property  and $X^*$ has the w$^*$-Kadec-Klee property.

	The paper is organized as follows. In Section \ref{sec2} we introduce the necessary background material which is well-known in principle \cite{Ba,Phe,Simons}.  Since the field requires a substantial amount of notation and concepts, we nevertheless provide some details.  In Section \ref{sec3} we consider a maximal monotone operator $A$ of type (D) in a nonreflexive Banach space whose dual space is strictly convex.  In Theorem \ref{T.s5.1} we represent  the value $Ax$ via the values of $A$ at nearby points of $x$. Theorem \ref{Conv.function} deals with lower semicontinuous convex functions. 
	In the final section we give a representation formula for the support function of the value of a maximal monotone operator in a nonreflexive Banach space whose dual space has the w$^*$-Kadec-Klee property by means of its minimal-norm  selections.

	\section{Preliminaries} \label{sec2}
	
	Throughout the present paper, $(X,\Vert \cdot \Vert)$ is a real Banach space and $X^*, X^{**}$ are its dual and bidual spaces,  respectively. The spaces $X$ and $X^*$ are paired by $\langle \cdot, \cdot\rangle$. We denote by $\rightarrow$, $\overset{\ope{w}}{\to}$, and $\overset{\ope{w}^*}{\rightharpoonup}$ the norm convergence, weak convergence, and weak$^*$ convergence of nets, respectively.  
	
	We first recall some geometric properties of real Banach spaces.
	\begin{definition}\rm Let $X$ be a real Banach space. 
		\begin{itemize}
			\item[(i)] $X$ is called  \emph{strictly convex} if for $x,y\in X$, $x \ne y$ with $\Vert x \Vert = \Vert y \Vert =1$, then $\Vert x+y\Vert <2.$
			\item[(ii)] $X$ is called  \emph{locally uniformly convex} if, for every $x \in X$ with $\Vert x\Vert =1$ and $\varepsilon >0$, there exists $\delta (x, \varepsilon)>0$ such that for any $y \in X$ with $\Vert y\Vert =1$, $\Vert x-y\Vert > \varepsilon$, then $ \Vert x+y\Vert < 2- \delta.$
		\end{itemize}
	\end{definition}
	It is clear that a locally uniformly convex space is strictly convex. 
	The \textit{duality mappings} of $X$ and $X^*$ are respectively defined by
	\begin{align}\label{D.M}
		J(x)&:= \{x^*\in X^*:\, \frac{1}{2}\Vert y\Vert^2 \ge \frac{1}{2}\Vert x\Vert^2 +\langle x^*, y-x \rangle, \forall y \in X\},\\ 
		\label{D.M.D}
		J^*(x^*)&:= \{x^{**} \in X^{**}: \langle x^{**}, x^*\rangle = \Vert x^*\Vert^2 = \Vert x^{**}\Vert^2\}.
	\end{align}
	
	For $x \in X$,  the set $J(x)$ is nonempty and closed convex. If $X^*$ is strictly convex, then $J(x)$ is a singleton.

	Let $C$ be a nonempty subset  of $X$. We denote by 
	$\overline{C}, \ope{co}C,\ope{int}C$, and $\ope{bd}C$ the \emph{closure}, \emph{convex hull},  \emph{interior}, and \emph{boundary} of $C$, respectively. For $x \in \overline{C}$, the \emph{normal cone} of $C$ at $x$ is defined by
	$$N_C(x):= \{x^* \in X^*: \langle x^*, y-x\rangle \le 0, \forall y \in C\}.$$
	It is well-known that $N_C(x)$ is closed convex.
	The \emph{support function} of $C$ is the function from $X^*$ to $\overline{\mathbb{R}}:= \mathbb{R} \cup \{+\infty\}$ defined by 
	$$ \sigma_{C}(x^*)= \underset{x \in C}{\sup}\,\langle x^*, x \rangle.$$
	The support function $\sigma_K: X \to \overline{\mathbb{R}}$ of a nonempty subset  $K \subset X^*$ is defined analogously. 
	
	Let $f: X \to \overline{\mathbb R}$  be a convex function. The \emph{domain} of $f$, denoted by $\operatorname{dom}f$, is the set of points $x \in X$ such that $f(x) <+\infty$. We call $f$  \emph{lower semicontinuous} on $X$ if for all  $x\in X$   $$ \underset{y\to x}{\lim\inf} f(y) \ge f(x).$$
	Let $f$ be a lower semicontinuous convex function on $X$, $x \in \operatorname{dom}f$ and $\varepsilon \ge 0$. A linear continuous function $x^*\in X^*$ is called a \emph{$\varepsilon$-subgradient} of $f$ at $x$ if  
	$$ f(y) \ge f(x) +\langle x^*, y-x \rangle -\varepsilon, \, \forall y \in X.$$
	The set of all $\varepsilon$-subgradients of $f$ at $x$ is denoted by $\partial_{\varepsilon} f(x)$ and is called the $\varepsilon$-{\em subdifferential} of $f$ at $x$. For the case of $f(x)= +\infty$ we write $\partial_{\varepsilon} f(x)= \varnothing.$ 
	
	As usual, $\partial_{0}f(x)$ is denoted again by $\partial f(x)$ and is called  the $\emph{subdifferential}$ of $f$ at $x$. In the case of $f(x)=\frac{1}{2}\Vert x\Vert^2$ for  $x \in X$, we denote by $J_{\varepsilon}(x)$  the $\varepsilon$-subdifferential of $f$ at $x$.
	From \eqref{D.M}, $J(x)=\partial f(x)$.
	
	Let $A: X \rightrightarrows X^{*}$ be a set-valued operator. The \emph{domain}, \emph{graph}, and \emph{range} of $A$ are given, respectively, by
	\begin{align*}
		{D}(A)&:= \{x \in X: Ax \ne \varnothing\},\\
		{Gr}(A)&:= \{(x,x^*) \in X \times X^*: x^* \in Ax\},\\
		{R}(A)&:= \underset{x \in X}{\bigcup} Ax. 
	\end{align*}
	\begin{definition} \rm 
		(i) An operator $A$ is called \emph{monotone} if 
		$$ \langle x^*-y^*, x-y\rangle \ge 0, \mbox{ for } (x, x^*), (y,y^*) \in {Gr}(A).$$
		(ii) A monotone operator $A$ is called  \emph{maximal monotone} if  ${Gr}(A)$ is not properly
		contained in the graph of any other monotone operators. 
	\end{definition} 
	
	If $A$ is a maximal monotone operator, then for every $x \in X,$ $Ax$ is closed, convex (see \cite[Proposition 2.1]{Ba}) and satisfies the equality
	\begin{equation}\label{2.1}
		Ax= Ax+ N_{\overline{{D}(A)}}(x).
	\end{equation}
	Let $x \in {D}(A)$. The  \emph{minimal-norm selection} of $Ax$ is the set  
	$$A^\circ x:= \{x^*\in Ax:  \Vert x^*\Vert= \underset{y^* \in Ax}{\inf}\Vert y^*\Vert \}.$$
	If $X^*$ is strictly convex, then $A^{\circ}$ is a singleton. 
	
	We call a net $(x_i)_{i \in I}$ in $X$ \emph{eventually bounded} if there exists $M>0$ and $i_0 \in I$ such that
	$$ \Vert x_i\Vert < M,\,\, \forall i\succeq_I i_0.$$ 
	According to \cite[Fact 3.5]{BY1}, if $(x_i, x^*_i)_{i \in I}$ is a net in ${Gr}(A)$ such that $x_i \to x$, $x^*_i \overset{\ope{w}^*}{\rightharpoonup} x^*$, and $(x^*_i)_{i \in I}$ is eventually bounded, then $(x, x^*) \in {Gr}(A)$.
	
	The next concept was introduced by  Gossez (\cite{Go}).
	\begin{definition}\rm  Let $X$ be a Banach space and $A: X \rightrightarrows X^{*}$ be a maximal monotone operator. The  Gossez’s monotone closure of $A$ is defined by 
		$$ \overline{A}^g= \Big\{(x^{**}, x^*)\in X^{**} \times X^*: \langle x^*-y^*, x^{**}-y\rangle \ge 0, \forall (y, y^*) \in {Gr}(A)\Big\}.$$
		The operator $A$ is called of   \emph{type (D)} if for any $(x^{**}, x^*) \in \overline{A}^g$ there exists a bounded net $(x_i, x^*_i)_{i\in I}$ in ${Gr}(A)$ such that $x_i$ converges to $x^{**}$ in $\tau(X^{**}, X^*)$ and $x^*_i$ strongly converges to $x^*$.
	\end{definition}
	
	The following theorem is a characterization of maximal monotone operators of type (D). 
	\begin{theorem}[{\cite[Theorem 4.4]{MaSv1}}]\label{D.eq} Let $X$ be a Banach space and $A: X \rightrightarrows X^{*}$ be a maximal monotone operator. Then $A$ is of type (D) if and only if
		$$ {R}(\lambda J_{\varepsilon}(\cdot -x) +A) =X^*,  \forall \lambda >0, \varepsilon >0, \forall x \in X.$$
	\end{theorem}

	Finally,  let $F: X \rightrightarrows X^{*}$ be a set-valued mapping. For $x\in X$ we define some types of convergent limits
	\begin{equation*}
		\begin{split}
			\ope{w}^*-\underset{y \to x, y\ne x}{\Lim\sup} \,Fy :=\Big\{  x^* & \in X^*:  \exists \text{ a net } (x_i, x^*_i)_{i\in I} \in {Gr}(F)  \text{ with }  \\
			&x_i \ne x ,\forall i \in I,  x_i \to x, \,\, x^*_i \overset{\ope{w}^*}{\rightharpoonup} x^*\Big\}, \\
			\ope{bw}^*-\underset{y \to x, y\ne x}{\Lim\sup} \,Fy :=\Big\{  x^* & \in X^*:  \exists \text{ a net } (x_i, x^*_i)_{i\in I} \in {Gr}(F) \text{ and } M >0 \text{ with }  \\
			&x_i \ne x ,\forall i \in I,  x_i \to x,  \Vert x^*_i\Vert \le M ,\, \forall  
			i\in I, \,\, x^*_i \overset{\ope{w}^*}{\rightharpoonup} x^*\Big\},\\
			{\ope{s}-}\underset{y \to x, y\ne x}{\Lim\sup} \,Fy:=\Big\{  x^* & \in X^*:  \exists \text{ a net } (x_i, x^*_i)_{i\in I} \in {Gr}(F)  \text{ with }  x_i \ne x ,\forall i \in I,\\
			& x_i \to x, \,  x^*_i \to x^*\Big\}.
		\end{split}
	\end{equation*}

	\section{Representation formulas for  maximal monotone operators of type (D)} \label{sec3}
	
	In this section we consider a maximal monotone operator $A$ of type (D) in a Banach space $X$
	whose dual space is strictly convex. We represent  the value $Ax$ via the values at nearby points of $x$. We define new properties for $X$ and $X^*$, which are generalizations of the Kadec-Klee property, named the w-Kadec-Klee and $\ope{w}^*$-Kadec-Klee, respectively.  A respective representation was obtained in the case that $X^*$  has the $\ope{w}^*$-Kadec-Klee property.

	We first observe a property of the minimal-norm selection.
	\begin{lemma}\label{mlm} Let $A: X \rightrightarrows X^*$ be a maximal monotone operator. Then the minimal-norm  $A^\circ x$ is nonempty for all $x \in {D}(A).$ 
	\end{lemma}
	\begin{proof} Let $x \in {D}(A)$. Then $\alpha:= \inf_{x^* \in Ax}\Vert x^*\Vert<+\infty$.   For every $n \in \mathbb{N}$, we take $x^*_n \in Ax$ such that 
		\begin{equation}\label{3.1}
			\Vert x^*_n\Vert \le  \alpha +\frac{1}{n}.
		\end{equation}
		By Banach-Alaoglu theorem, there exist a subnet $(x^*_i)_{i \in I}$ of $(x^*_n)_{n \in \mathbb{N}}$ and $x^* \in X^*$ such that $x^*_i \overset{\ope{w}^*}{\rightharpoonup} x^*.$ It follows from inequality \eqref{3.1} that $\Vert x^*\Vert \le \displaystyle\liminf_{i }\Vert x^*_i\Vert \le \alpha.$ On the other hand, also from inequality \eqref{3.1}, the subnet $(x^*_i)_{i \in I}$ is eventually bounded. Owing to \cite[Section 2]{BFG}, we have $x^* \in Ax$. This implies $\Vert x^* \Vert \ge \alpha$ and hence $\Vert x^*\Vert = \alpha$,  so we get $x^* \in A^\circ x$, proving the lemma.
	\end{proof}
	The following properties of  $J_{\varepsilon}$ are useful later. 
	\begin{lemma}\label{L.1} Let $x \in X$ and $\varepsilon \ge 0, \lambda >0$. Then 
		\begin{itemize}
			\item[(i)] $J_{\varepsilon}(-x) = -J_{\varepsilon}(x), \,\, \lambda J_{\varepsilon}(x)= J_{\lambda^2\varepsilon}(\lambda x);$
			\item[(ii)] $\underset{x^* \in J_{\varepsilon}(x)}{\sup}\vert \Vert x\Vert -\Vert  x^*\Vert\vert \le \sqrt{2\varepsilon};$
			\item[(iii)]  $\underset{x^* \in J_{\varepsilon}(x)}{\sup}\vert \langle x^*, x \rangle -\Vert x\Vert^2 \vert \le  \sqrt{\varepsilon}(1+ \frac{1}{2}\Vert x\Vert^2).$
		\end{itemize} 
	\end{lemma}
	
	\begin{proof} (i) Let $x \in X, \varepsilon \ge 0, \lambda >0$ and $x^* \in J_{\varepsilon}(-x).$
		It follows from the definition of $J_{\varepsilon}$ that
		\begin{equation}\label{eq1} \frac{1}{2}\Vert y\Vert^2 + \varepsilon \ge \frac{1}{2}\Vert x\Vert^2 + \langle  x^*, y+x \rangle, \, \forall y \in X.
		\end{equation}
		This implies 
		$$ \frac{1}{2}\Vert y\Vert^2 + \varepsilon \ge \frac{1}{2}\Vert x\Vert^2 + \langle - x^*, y -x \rangle, \forall y \in X,$$
		which gives $-x^* \in J_{\varepsilon}(x)$. Hence, we have the inclusion 
		$$J_{\varepsilon}(-x) \subset -J_{\varepsilon}(x), \forall x \in X.$$ 
		Switching the roles of $x$ and $-x$, we also get $-J_{\varepsilon}(x)\subset J_{\varepsilon}(-x)$, implying the former. For the latter, multiplying both side of \eqref{eq1} by $\lambda^2$, we get
		\begin{equation*}
			\frac{1}{2}\Vert \lambda y\Vert^2 +\lambda^2\varepsilon \ge \frac{1}{2}\Vert \lambda x\Vert^2 + \langle \lambda x^*, \lambda y - \lambda x\rangle, \forall y \in X.
		\end{equation*}
		This yields 
		\begin{equation}
			\frac{1}{2}\Vert y\Vert^2 +\lambda^2\varepsilon \ge \frac{1}{2}\Vert \lambda x\Vert^2 + \langle \lambda x^*, y - \lambda x\rangle, \forall y \in X.
		\end{equation}
		Then  $\lambda x^* \in J_{\lambda^2\varepsilon}(\lambda x)$ implies that
		\begin{equation}\label{c.1}
			\lambda J_{\varepsilon}(x) \subset J_{\lambda^2\varepsilon}(\lambda x), \,\, \forall x \in X, \forall \varepsilon \ge 0, \forall \lambda >0.
		\end{equation} 
		This leads to
		\begin{equation}\label{c.2}
			\frac{1}{\lambda}J_{\lambda^2\varepsilon}(\lambda x) \subset J_{(\frac{1}{\lambda})^2(\lambda^2\varepsilon)}(\frac{1}{\lambda}\lambda x)= J_{\varepsilon}(x),\,\, \forall x \in X,\forall \varepsilon \ge 0, \forall \lambda >0.
		\end{equation}
		Combining \eqref{c.1} and \eqref{c.2}, we get the latter. 
		
		(ii) This follows from 
		$$ \vert\Vert x^*\Vert - \Vert x\Vert\vert \le \sqrt{2\varepsilon}$$ due to 
		\cite{MaSv}.
		
		(iii) This is obvious for $\varepsilon=0.$ For $\varepsilon >0$ the definition of $J_{\varepsilon}$ implies that 
		$$ \frac{1}{2}\Vert (1+t)x\Vert^2 \ge \frac{1}{2}\Vert x\Vert^2 + t\langle x^*, x\rangle - \varepsilon,\,\,\forall t \in \mathbb{R}.$$
		This is equivalent to
		$$ t(\langle x^*,x\rangle - \Vert x\Vert^2) \le \varepsilon +  \frac{t^2}{2}\Vert x \Vert^2,\,\, \forall t \in \mathbb{R}.$$
		In particular, for $t= \sqrt{\varepsilon}$ we get
		$$ \vert \langle x^*,x\rangle - \Vert x\Vert^2 \vert \le \sqrt{\varepsilon}\left(1+\frac{1}{2}\Vert x \Vert^2\right),$$
		which proves (iii).
	\end{proof}
	
	The following lemma is an extension  of  \cite[Proposition 2.2]{Ba} 
	on the Yosida approximation of maximal monotone operators in reflexive Banach spaces.
	\begin{lemma}\label{L.KEY}
		Let $X$ be a Banach space whose dual space is strictly convex, $A: X \rightrightarrows X^*$ be a maximal monotone operator of type (D). 	
		For every $x \in {D}(A)$ and a sequence $\lambda_n \downarrow 0$ there exist
		a subnet $(\lambda_i)_{i\in I}$ of $(\lambda_n)_{n \in \mathbb{N}}$, a net $(x_i)_{i \in I} \subset {D}(A)$ converging to $x$,  a net $(z^*_i)_{i\in  I} \subset X^*$ with $z^*_i \in J_{\lambda^6_i}(x-x_i)$ such that
		$(\lambda^{-1}_iz^*_i)_{i \in I}$ is bounded and
		\begin{equation}\label{KEY}
			\lambda^{-1}_iz^*_i \in Ax_i, \,\,\, \, \Vert \lambda^{-1}_iz^*_i\Vert \to  \Vert A^\circ x\Vert, \,\,\, \lambda_i^{-1} z^*_i \overset{\ope{w}^*}{\rightharpoonup} A^\circ x. 
		\end{equation}
	\end{lemma}
	\begin{proof} Let $x \in {D}(A)$ and a sequence $\lambda_n \downarrow 0$. By Theorem \ref{D.eq}, for every $n \in \mathbb{N}$ we can find $x_n \in X$ such that 
		$$ 0 \in \lambda^{-1}_nJ_{\lambda^6_n}(x_n-x) + Ax_n.
		$$
		By Lemma \ref{L.1}~(i), there exists a sequence $(z^*_n)_{n\in \mathbb{N}} \subset X^*$ such that $z^*_n \in J_{\lambda^6_n}(x-x_n)$ and
		\begin{equation}\label{l.4.2.1}
			\lambda^{-1}_nz^*_n \in Ax_n,\,\, \forall n\in \mathbb{N}.
		\end{equation}
		Due to the monotonicity of $A$ and Lemma \ref{mlm}, the set $A^\circ(x)$ is nonempty and
		\begin{equation}\label{a0} \langle \lambda^{-1}_nz^*_n -A^\circ x, x -x_n\rangle \le 0, 
		\end{equation}
		which implies that 
		\begin{equation}\label{l.4.2.1.1}
			\lambda^{-1}_n\langle z^*_n, x-x_n\rangle \le \langle A^\circ x, x- x_n\rangle \le \Vert A^\circ x\Vert\left\Vert x_n-x\right\Vert.
		\end{equation}
		Since $z^*_n \in J_{\lambda^6_n}(x-x_n)$, Lemma \ref{L.1}~(iii) gives us
		$$ \langle z^*_n, x-x_n \rangle \ge \Vert x-x_n\Vert^2-\lambda^3_n\left(1+\frac{1}{2}\Vert x-x_n\Vert^2\right) \ge \Vert x-x_n\Vert^2-\lambda^3_n(1+\Vert x-x_n\Vert^2).$$ 
		It follows from \eqref{l.4.2.1.1} that
		$$ \lambda^{-1}_n\Big(\Vert x-x_n\Vert^2 -\lambda^3_n( 1+\Vert x-x_n\Vert^2) \Big) \le  \Vert A^\circ x\Vert\left\Vert x-x_n\right\Vert.$$
		Multiplying both sides by $\lambda^{-1}_n$ and then adding $\lambda_n$, we get
		\begin{equation*}
			(1-\lambda^3_n)(\lambda^{-1}_n\Vert x-x_n\Vert)^2 \le \Vert A^\circ x\Vert(\lambda^{-1}_n\Vert x-x_n\Vert) + \lambda_n,
		\end{equation*}
		which yields
		\begin{equation}\label{l.4.2.2.1}
			\lambda^{-1}_n\Vert x-x_n\Vert \le \frac{\Vert A^\circ x \Vert + \sqrt{\Vert A^\circ x \Vert^2 +4\lambda_n(1-\lambda^3_n)}}{2(1-\lambda^3_n)}
		\end{equation}
		for sufficiently large $n$, 
		and hence
		\begin{equation}\label{l.4.2.2}
			\underset{n \to \infty}{\lim\sup}\,\, \lambda^{-1}_n\Vert x-x_n \Vert \le \Vert A^\circ x\Vert.
		\end{equation}
		Furthermore, by Lemma \ref{L.1}~(ii),  
		$$ \Vert z^*_n\Vert \le \Vert x-x_n\Vert +\sqrt{2}\lambda^3_n, \,\,\forall n\in \mathbb{N}. $$ 
		Together with  \eqref{l.4.2.2}, we get 
		\begin{equation}\label{l.4.2.4}
			\underset{n \to \infty}{\lim\sup}\,\, \Vert \lambda^{-1}_n z^*_n \Vert \le \underset{n \to \infty}{\lim\sup}\,\,(\lambda^{-1}_n\Vert x-x_n\Vert+\sqrt{2}\lambda^2_n) \le \Vert A^\circ x\Vert.
		\end{equation}
		This shows that $(\lambda^{-1}_nz^*_n)_{n \in \mathbb{N}}$ is bounded. Using Banach-Alaoglu theorem, there exist $y^* \in X^*$ and a subnet $(\lambda^{-1}_i z^*_i)_{i \in I}$ of $(\lambda^{-1}_n z^*_n)_{n \in \mathbb{N}}$ satisfying
		\begin{equation}\label{l.4.2.5}
			\lambda^{-1}_i z^*_i \overset{\ope{w}^*}{\rightharpoonup} y^*.
		\end{equation}
		It follows from \eqref{l.4.2.4} that 
		\begin{equation}\label{l.4.2.6}
			\Vert y^*\Vert \le \underset{i}{\lim\inf}\,\Vert \lambda^{-1}_i z^*_i\Vert \le \underset{i}{\lim\sup}\,\Vert \lambda^{-1}_i z^*_i\Vert  \le \Vert A^\circ x\Vert.
		\end{equation}
		On the other hand, owing to  \eqref{l.4.2.2.1},  $x_n \to x \text{ as } n \to \infty$, and hence
		$ x_i\to x.$ It follows from  \eqref{l.4.2.4} and \eqref{l.4.2.5} that $(\lambda^{-1}_iz^*_i)_{i\in I}$ is bounded and  
		\begin{equation}\label{l.4.2.3.1}
			\lambda^{-1}_iz^*_i \in Ax_i, \lambda^{-1}_iz^*_i \overset{\ope{w}^*}{\rightharpoonup} y^*, \,\,  x_i \to x.
		\end{equation}
		Then $y^* \in Ax$ by the maximal
		monotonicity of $A$. From \eqref{l.4.2.6}, we get $y^*= A^\circ x$ since $A^\circ x$ is a singleton set due to the strict convexity of $X^*$. Hence, we obtain from \eqref{l.4.2.6} that $ \Vert \lambda^{-1}_i z^*_i\Vert \to \Vert A^\circ x\Vert$, completing the proof.
	\end{proof}
	Recall that if $A$ is a maximal monotone operator, then $Ax$ is convex and closed in $X^*$ for $x\in{D}(A)$. In order to represent $Ax$ by means of the values at nearby points of $x$, we establish a relation between the values at nearby points of $x$ and the faces $A(x;v)$ of $Ax$ with respect to vector $v \in X \setminus \{0\},$ which is defined by
	$$ A(x;v):= \Big\{y^*\in Ax: \langle y^*, v\rangle= \underset{x^* \in Ax}{\sup} \langle x^*, v \rangle\Big\}.$$ This relation depends on the  properties of  $X$.  
	Next we introduce a new class of Banach spaces. 
	\begin{definition}\rm Let $X$ be a Banach space.  
		\begin{itemize} 
			\item[(a)] $X$ is said to have the {\em w-Kadec-Klee property} if for  any net $(x_i)_{i \in I} $, $x_i \overset{\ope{w}}{\to} x$ and $\Vert x_i\Vert \to \Vert x\Vert$, then $x_i \to x$.
			\item [(b)]  $X^*$ is said to have the  {\em $w^*$-Kadec-Klee property} if for any net $(x^*_i)_{i\in I} $,  $x^*_i \overset{\ope{w}^*}{\rightharpoonup} x^*$  and $\Vert x^*_i\Vert \to \Vert x^*\Vert$, then $x^*_i \to x^*$.
		\end{itemize}
	\end{definition}
	The following lemma shows that the class of Banach spaces with $w^*$-Kadec-Klee property
	includes locally uniformly convex Banach spaces. 
	\begin{lemma}\label{l.m.t} Let $X$ be a Banach space. Then,
		
		(a) If $X$ is locally uniformly convex, then $X$ has w-Kadec-Klee property;
		
		(b) If $X^*$ is locally uniformly convex, then $X^*$ has $w^*$-Kadec-Klee property.
	\end{lemma} 
	See \cite[Proposition 3.7.6]{Za} for a proof of (a). The proof of (b) is similar.
	
	The next result discovers a relation between the faces of $Ax$ 
	and the set consisting of the weak$^*$ convergent limits of bounded nets of $A$ at nearby points of $x$.
	In what follows  $j_{\lambda}(w)$ means an element of  $J_{\lambda}(w)$. 
	\begin{proposition}\label{T.s4.1}
		Let $X$ be a Banach space whose dual space is  strictly convex and let $A: X \rightrightarrows X^*$ be a maximal monotone operator of type (D). Then
		for  $x \in {D}(A)$ and $v \in X \setminus \{0\}$ 
		\begin{equation}\label{T.4.1.1}
			A(x;v) \subset \ope{bw}^*{-}\underset{\lambda \to 0\atop \underset{j_{\lambda}(w) \overset{\ope{w}^*}{\rightharpoonup} J(v)}{\Vert j_{\lambda}(w)\Vert \to \Vert J(v)\Vert}}{\Lim\sup}\,\, A(x+\lambda w).
		\end{equation}
		If in addition that $X^*$ has the $w^*$-Kadec-Klee property, then for $x \in {D}(A)$ and $v \in X \setminus \{0\}$
		\begin{equation}\label{T.4.1.2}
			A(x; v) \subset \ope{s}-\underset{\lambda \to 0 \atop j_{\lambda}(w) \to J(v)}{\Lim\sup} A(x+\lambda w).
		\end{equation}	
	\end{proposition}
	\begin{proof}Let $x \in {D}(A)$, $v \in X \setminus \{0\}$ and $x^* \in A(x;v).$ Since $X^*$ is strictly convex,  $J(v)$ is a singleton set. We consider a maximal monotone operator defined by
		$$ By:= Ay-J(v)-x^*, \,\, \forall \, y \in X.$$
		It is obvious that ${D}(B)= {D}(A)$, and hence $x \in {D}(B)$. We first claim that 
		\begin{equation}\label{t.4.1.0}
			B^\circ x= -J(v).
		\end{equation}
		Indeed, observe that
		\begin{equation}\label{t.4.1.1} 
			-J(v) = x^*-J(v)-x^* \in Ax-J(v)-x^* = Bx,
		\end{equation}
		Moreover, for $y^* \in Ax$, $\langle x^*-y^*, v\rangle \ge 0$ since  $x^* \in A(x;v)$. This implies
		\begin{equation*}
			\begin{split}
				\Vert -J(v)\Vert & = \Vert v\Vert^{-1}\langle J(v), v \rangle  \\
				& \le \Vert v\Vert^{-1}\langle J(v)+x^*-y^*, v \rangle \\
				& \le \Vert v\Vert^{-1}\Vert J(v)+x^*-y^*\Vert \Vert v\Vert \\
				& = \Vert y^*-J(v)-x^*\Vert.
			\end{split}
		\end{equation*}
		Now \eqref{t.4.1.0} follows from \eqref{t.4.1.1}.
		
		Next, we apply  Lemma \ref{L.KEY} to have nets $(\lambda_i)_{i\in I} \to 0$, $(x_i)_{i\in I} \subset {D}(B)$ with $x_i \to x$ and $(z^*_i)_{i \in I} \subset X^*$ with $ z^*_i \in J_{\lambda^6_i}(x-x_i)$ such that  $(\lambda^{-1}_iz^*_i)_{i \in I}$ is bounded  
		and \begin{equation*}\label{e.b.2}
			\lambda^{-1}_iz^*_i \in Bx_i, \,\, \,\, \lambda^{-1}_iz^*_i \overset{\ope{w}^*}{\rightharpoonup} B^\circ x = -J(v), \Vert \lambda^{-1}_iz^*_i\Vert \to \Vert J(v) \Vert.
		\end{equation*}
		Define $w_i:= \lambda^{-1}_i(x_i-x), i \in I$. We then have
		$$ (\lambda^{-1}_iz^*_i + J(v)+ x^*)_{i \in I}\,\,  \text{ is bounded},$$
		\begin{equation}\label{t.4.1.3}
			\lambda^{-1}_iz^*_i+ J(v)+ x^* \in Bx_i+J(v)+x^* = Ax_i= A(x+\lambda_iw_i),
		\end{equation}
		\begin{equation}\label{t.4.1.4}
			\lambda^{-1}_iz^*_i+J(v)+ x^* \overset{\ope{w}^*}{\rightharpoonup} x^*,
		\end{equation}
		\begin{equation}\label{t.4.1.7} 
			-\lambda_i^{-1}z^*_i \overset{\ope{w}^*}{\rightharpoonup} J(v), \Vert -\lambda_i^{-1}z^*_i\Vert \to \Vert J(v) \Vert.
		\end{equation} 	
		It follows from Lemma \ref{L.1}~(i) that 
		\begin{equation}\label{t.4.1.6} 
			-\lambda^{-1}_iz^*_i \in J_{\lambda^4_i}(\lambda^{-1}_i(x_i-x)) \subset J_{\lambda_i}(\lambda^{-1}_i(x_i-x)) = J_{\lambda_i}(w_i).
		\end{equation}
		Thus 
		$$ x^* \in {\rm bw}^*-\underset{t \to 0\atop \underset{j_t(w) \overset{\ope{w}^*}{\rightharpoonup} J(v)}{\Vert j_t(w)\Vert \to \Vert J(v)\Vert}}{\Lim\sup}\,\, A(x+tw),$$
		and hence \eqref{T.4.1.1} holds.
		
		Now, suppose in addition that $X^*$ has the $\ope{w}^*$-Kadec-Klee property and $x^* \in A(x;v)$ with $x \in {D}(A)$ and $v \ne 0$. By the above result, we find nets $(\lambda_i)_{i \in I} \subset\mathbb{R}_{+},$ $(x_i)_{i \in I}\subset {D}(A)$, $(w_i)_{i\in I} \subset X$ and $(z^*_i)_{i\in I} \in X^*$ satisfying \eqref{t.4.1.3}-\eqref{t.4.1.7}. It follows from \eqref{t.4.1.7} that $-\lambda_i^{-1}z^*_i \to J(v)$, and hence $\lambda^{-1}_iz^*_i+J(v)+ x^* \to x^*.$ This leads to
		$$ x^* \subset \ope{s}-\underset{t \to 0 \atop j_{t}(w) \to J(v)}{\Lim\sup} A(x+tw),$$ completing the proof.
	\end{proof}
	\begin{corollary}\label{C.4}	Let $X$ be a Banach space whose dual space is  strictly convex and let $A: X \rightrightarrows X^*$ be a maximal monotone operator of type (D).
		Then, for $x \in {D}(A)$ and $v \in X \setminus \{0\}$ 
		\begin{equation}\label{C.4.1.1}
			A(x;v) \subset {\rm bw}^*-\underset{y \to x, y\ne x}{\Lim\sup}\,Ay.
		\end{equation}
		If in addition that $X^*$ has the ${w}^*$-Kadec-Klee property, then for $x \in {D}(A)$ and $v \in X \setminus \{0\},$
		\begin{equation}\label{C.4.1.2}
			A(x; v) \subset \ope{s}-\underset{y \to x, y\ne x}{\Lim\sup}\,Ay.
		\end{equation}
	\end{corollary}
	\begin{proof} Let $x \in {D}(A)$, $v \in X \setminus \{0\}$ and $x^* \in A(x;v)$. By \eqref{T.4.1.1}, there exist nets $(\lambda_i)_{i \in I} \to 0,$ $(w_i)_{i \in I} \subset X,$ $(z^*_i)_{i\in I} \subset X^*$ with $z^*_i \in J_{\lambda_i}(w_i)$, $z^*_i \overset{\ope{w}^*}{\rightharpoonup} J(v),$ $\Vert z^*_i \Vert \to \Vert J(v)\Vert$, $(x^*_{i})_{i\in I} \subset X^*$ with $x^*_i \in A(x+\lambda_iw_i)$, such that
		$$(x^*_i)_{i\in I} \text{ is bounded and } x^*_i  \overset{\ope{w}^*}{\rightharpoonup}  x^*.$$
		Next we verify that $x+\lambda_iw_i \to x$, which would imply 
		$$ x^* \in {\rm bw}^*-\underset{y \to x, y\ne x}{\Lim\sup}\,Ay$$
		and hence  \eqref{C.4.1.1} holds.  By Lemma \ref{L.1}~(ii) and $z_i \in J_{\lambda_i}(w_i),$ we have $\Vert w_i\Vert \le \Vert z^*_i \Vert +\sqrt{2\lambda_i}$ for $i \in I$. Therefore 
		$$ \underset{i\in I}{\limsup}\Vert w_i\Vert \le \underset{i\in I}{\sup}(\Vert z^*_i\Vert + \sqrt{2\lambda_i}) = \Vert J(v)\Vert= \Vert v\Vert$$
		implies $x+\lambda_iw_i \to x.$
		The same argument applies to yield  \eqref{C.4.1.2}.
	\end{proof}
	
	Now we are in a position to derive the first representation for the value of a maximal monotone operator of type (D) at a point in its domain.
	\begin{theorem}\label{T.s5.1}
		Let $X$ be a Banach space whose dual space is  strictly convex, and let $A: X \rightrightarrows X^*$ be a maximal monotone operator  of type (D).
		Then for $x\in X$
		\begin{equation}\label{**}
			Ax = \ope{co}\overline{\{\ope{bw}^*-\underset{y \to x, y\ne x}{\Lim\sup}\,Ay\}} + N_{\overline{{D}(A)}}(x).
		\end{equation}
		If in addition that $X^*$ has the ${w}^*$-Kadec-Klee property, then for $x\in X$
		\begin{equation}\label{***}
			Ax = \ope{co}\{\ope{s}-\underset{y \to x, y\ne x}{\Lim\sup}\,Ay\} + N_{\overline{{D}(A)}}(x).
		\end{equation}
	\end{theorem}
	\begin{proof} Let $x \in X$ and 
		$$K:= \ope{bw}^*-\underset{y \to x, y\ne x}{\Lim\sup}\,Ay.$$ 
		Owing to  the maximal monotonicity of $A$,  $K \subset Ax$ and the fact that $Ax$ is closed convex and satisfies  $Ax= Ax+ N_{\overline{{D}(A)}}(x)$, we deduce the inclusion
		$$  \ope{co}\overline{\{\ope{bw}^*-\underset{y \to x, y\ne x}{\Lim\sup}\,Ay\}} + N_{\overline{{D}(A)}}(x) \subset Ax.$$
		To justify the reverse inclusion, we will show the following two inclusions 
		\begin{equation}\label{t.5.1.1}
			Ax \subset \ope{co}(\ope{bd}Ax) + N_{\overline{{D}(A)}}(x)
		\end{equation}
		and
		\begin{equation}\label{t.5.1.2}
			\ope{bd}Ax \subset \overline{K}.
		\end{equation}
		For the case $x \notin {D}(A)$, one has  $K= Ax = \ope{bd}Ax= \varnothing$ and so the above inclusions  are trivial. Let $x \in {D}(A).$ To clarify \eqref{t.5.1.1}, we pick any $x^* \in Ax$. If $x^* \in \ope{bd}Ax$, then  
		\begin{equation}\label{t.5.1.3}
			x^* \in \ope{bd}Ax \subset \ope{co}(\ope{bd}Ax) + N_{\overline{{D}(A)}}(x).
		\end{equation}
		If $x^* \in \ope{int}Ax$, we take  $x^*_0 \in \ope{bd}Ax$ and define
		$$ \bar{\rho}:= \sup\big\{ \rho > 0: x^* +t(x^*-x^*_0) \subset Ax, \forall t \in [0, \rho)\big\}>0.$$
		We consider two cases of $\bar{\rho}$.
		
		\textit{Case 1:} $\bar{\rho} = +\infty$. It means that $x^*+ t(x^*-x^*_0) \in Ax$ for every $t \ge 0$. Hence, for $(y, y^*) \in {Gr}(A)$ by the monotonicity of $A$, we have
		$$ \langle x^*+ t(x^*-x^*_0)-y^*, y-x \rangle \le 0, \forall t \ge 0.$$
		This implies $\langle x^*-x^*_0, y-x \rangle \le 0$. Since this holds for any $y \in {D}(A),$ we get $x^*-x^*_0 \in N_{\overline{{D}(A)}}(x)$, which implies that
		$$ x^* \in x^*_0 + N_{\overline{{D}(A)}}(x) \subset \ope{co}(\ope{bd}Ax) + N_{\overline{{D}(A)}}(x).$$
		
		\textit{Case 2:} $\bar{\rho} < +\infty.$ Since $Ax$ is closed,  $z^*:= x^* + \bar{\rho}(x^*-x^*_0) \in Ax.$ Furthermore, $z^* \in \ope{bd}Ax$ by the definition of $\bar{\rho}$. Therefore
		$$ x^* = \frac{1}{1+\bar{\rho}}z^*+ \frac{\bar{\rho}}{1+ \bar{\rho}}x^*_0 \in \ope{co}(\ope{bd}Ax) \subset \ope{co}(\ope{bd}Ax) + N_{\overline{{D}(A)}}(x).$$   
		Combining the results above, we conclude that \eqref{t.5.1.1} holds.
		
		Finally, let $x^* \in \ope{bd}Ax$. According to \cite[Theorem 1]{Phel}, since $Ax$ is weak$^*$ closed convex, there exists a sequence $(x^*_n)_{n \in \mathbb{N}} \subset \ope{bd}Ax$ such that $x^*_n \to x^*$, where $x^*_n$ are weak$^*$ support points of $Ax$, i.e., for each $x^*_n$, there exists $v_n \in X \setminus \{0\}$ such that  
		$$\langle x^*_n,v_n\rangle = \sigma_{Ax}(v_n).$$	
		The equality means that $x^*_n \in A(x; v_n)$ for  $n \in \mathbb{N}$. By Corollary \ref{C.4},  $x^*_n \in K$, which implies $x^* \in \overline{K}$. Hence we get \eqref{t.5.1.2} and the proof of \eqref{**} is complete. 
		
		Now suppose in addition that $X^*$ has the $\ope{w}^*$-Kadec-Klee property. The proof of \eqref{***} is similar  to that of \eqref{**}, where $K$ is replaced by the closed set $ K':= \ope{s}-\underset{y \to x, y\ne x}{\Lim\sup}\,Ay.$
	\end{proof}
	\begin{remark} \rm With the setting in the previous theorem, for $x\in X$ 
		\begin{equation}\label{Im.remark}
			Ax = \ope{co}\,\overline{\{\ope{bw}^*-\underset{y \to x, y\ne x}{\Lim\sup}\,Ay\}} + N_{\overline{{D}(A)}}(x) \subset \ope{co}\overline{\{\ope{w}^*-\underset{y \to x, y\ne x}{\Lim\sup}\,Ay\}} + N_{\overline{{D}(A)}}(x).
		\end{equation}
		It is worth noting that in the case that $X$ is a Hilbert space, the above inclusion is still strict; see \cite[Example 1]{BFG}. Furthermore, the inclusion  becomes equality if the nets are replaced by sequences, and by \cite[Theorem 3.3]{HB}
		$$ Ax=  \ope{co}{\{\ope{ss}-\underset{y \to x, y \ne x}{\Lim\sup}\,Ay\}} + N_{\overline{{D}(A)}}(x),$$
		where 
		\begin{align*}
			{\ope{ss}-}\underset{y \to x, y\ne x}{\Lim\sup} \,Ay=\Big\{  x^*  \in X^*:  &\exists\,  (x_n, x^*_n)_{n \in \mathbb{N}} \in {Gr}(A)  \text{ with }  x_n \ne x ,\forall n \in \mathbb{N},\\
			& x_n \to x, \,  x^*_n \to x^*\Big\}.
		\end{align*}

		{\hfill $\Diamond$}
	\end{remark}

	The next theorem is a representation for maximal monotone operators in reflexive spaces. 
	\begin{theorem}\label{Refle.} Let $X$ be a reflexive Banach space and let $A: X \rightrightarrows X^*$ be a maximal monotone operator. Then for $x\in X$
		$$ Ax= \ope{co}\{\ope{s}-\underset{y \to x, y\ne x}{\Lim\sup}\,Ay\} + N_{\overline{{D}(A)}}(x).$$
	\end{theorem}
	\begin{proof}Let $x \in X$. We first observe that the maximal monotonicity of $A$ and the terms 
		$\ope{co}\{\ope{s}-\underset{y \to x, y\ne x}{\Lim\sup}\,Ay\}, N_{\overline{{D}(A)}}(x)$
		are not affected by equivalently renorming of $X$ and $X^*.$ By Troyanski's renorming theorem (\cite[Theorem 5.192]{GP}), $X$ admits an equivalent norm so that both $X$ and $X^*$ are locally uniformly convex. Therefore, we may assume that both $X$ and $X^*$ are locally uniformly convex. Then $X^*$ has the $\ope{w}^*$-Kadec-Klee property due to Lemma \ref{l.m.t}. Since the space $X$ is reflexive, the operator $A$ is of type (D), owing to  \cite[Proposition 1]{Rock.2}. Using the second part of Theorem \ref{T.s5.1}, we  derive the conclusion of the theorem. 
	\end{proof}

	The next result is a representation formula for the subdifferentials of a lower semicontinuous convex function. 
	\begin{theorem}\label{Conv.function}
		Let $X$ be a Banach space whose dual space is  strictly convex and let $f: X \to \overline{\mathbb{R}}$ be a lower semicontinuous function. Then for $x\in X$
		\begin{equation}\label{c.5.2.1}
			\partial f(x)= \ope{co}\overline{\{\ope{bw}^*-\,\,\underset{y \to x, y\ne x}{\Lim\sup}\,\partial f(y)\}} +N_{\overline{\ope{dom}f}}(x).
		\end{equation}
		If in addition that $X^*$ has the ${w}^*$-Kadec-Klee property, then for $x\in X$
		\begin{equation}\label{c.5.2.2} 
			\partial f(x)= \ope{co}\{\ope{s}-\,\,\underset{y \to x, y\ne x}{\Lim\sup}\,\partial f(y)\} +N_{\overline{\ope{dom}f}}(x).
		\end{equation}
	\end{theorem}
	
	\begin{proof} Using \cite[Lemma 1]{Go} and \cite[Corollary 2]{Ra} one has
		$$ {R}(\lambda J_{\varepsilon}(\cdot-x) + \partial f) = X^*, \forall x \in X,\forall \lambda, \varepsilon >0.$$
		Now the theorem follows from Theorem \ref{T.s5.1}. 
	\end{proof}
	\begin{remark}\rm
		Let $X$ be a Banach space. If the closed unit ball of $X^*$ is $\rm{weak}^*$ sequentially compact, then the nets in Lemma \ref{L.KEY} can be replaced by sequences. 
		As a consequence, all results in Proposition \ref{T.s4.1}, theorems \ref{T.s5.1}, \ref{Refle.}, \ref{Conv.function} and 
		Corollary \ref{C.4} are also true for sequences. In particular, these results are true for the case that $X$ is a weak Asplund space; see \cite[p.239]{Di}. 
		
		\hfill{$\Diamond$}
	\end{remark}
	\section{A representation formula for the support functions of the values of maximal monotone operators}

	In this section we consider a maximal monotone operator $A$ of type (D) in a Banach space which has the w-Kadec-Klee property and whose dual space has the w$^*$-Kadec-Klee property. We provide an explicit formula for the support function of $Ax$ by use of the minimal-norm selection of $Ay$ for nearby points $y$ of $x$.

	We first present a technical result for the theorem below.
	\begin{lemma}\label{L.s4.3}
		Let $X$ be a Banach space such that $X^*, X^{**}$ are strictly convex and $X$ has the w-Kadec-Klee property, $X^*$ has the $w^*$-Kadec-Klee property. Let $v \in X \setminus \{0\}$,  $(t_i)_{i \in I} \subset \mathbb{R}_{+}$ with $t_i \to 0$ and $(w_i)_{i \in I} \subset X$.  If there exists a net $(z^*_i)_{i \in I} \subset X^*$ with $z^*_i \in J_{t_i}(w_i)$, 
		$z^*_i \to J(v)$, then there exists a subnet of $(w_i)_{i \in I}$ converging to $v$. 
	\end{lemma}
	\begin{proof}
		Assume that $(z^*_i)_{i \in I}\subset X^*$ with $z^*_i \in J_{t_i}(w_i)$ converging to $J(v)$. 
		According to \cite[p.608]{BroRoc}, for every $i \in I$, there exists $v_i \in X$ such that
		\begin{equation}\label{l.4.3.1}
			\Vert w_i-v_i\Vert \le \sqrt{t_i}, \Vert z^*_i- J(v_i) \Vert \le \sqrt{t_i}.
		\end{equation}
		This implies $J(v_i) \to J(v),$ and hence $\Vert v_i\Vert \to \Vert v\Vert$, which shows that $(v_i)_{i \in I} \subset X \subset X^{**}$ is eventually bounded. Using Banach-Alaoglu theorem, by passing to a subnet, if necessary, that  $v_i \overset{\ope{w}}{\to} z^{**} \in X^{**}$. Furthermore, since $X^{**}$ is strictly convex,  we have $v_i= J^*(J(v_i))$ and $v= J^*(J(v))$.  By the maximal monotonicity of $J^*$, we deduce $z^{**} = J^*(J(v))$ and hence 
		$$ \Vert v_i\Vert \to \Vert v\Vert,\,\,\text{ and }\,\,  v_i \overset{\ope{w}}{\rightarrow} v. $$ 
		Owing to  the assumption that $X$ has the $\ope{w}$-Kadec-Klee property, we get $v_i \to v$. It follows from \eqref{l.4.3.1} that $w_i \to v$, which completes the proof.
	\end{proof}
	
	The following theorem represents the support function of the value $Ax$ via its minimal-norm selections.

	\begin{theorem}\label{T.s4.2}
		Let $X$ be a Banach space such that $X^*, X^{**}$ are strictly convex and $X$ has the w-Kadec-Klee property, $X^*$ has the $w^*$-Kadec-Klee property. Let $A: X \rightrightarrows X^*$ be a maximal monotone operator of type (D). Then for $x \in {D}(A), \, v \in X \setminus \{0\}$
		\begin{equation}\label{*}
			\sigma_{Ax}(v)= 
			\underset{t \downarrow 0 \atop w \to v}{\lim\inf}\langle A^{\circ}(x+tw),w\rangle. 
		\end{equation}	
	\end{theorem}
	\begin{proof} Let $x \in {D}(A)$ and $v \in X\setminus\{0\}$. We first show that 
		\begin{equation}\label{t.4.2.0}
			\sigma_{Ax}(v)\le \underset{t \downarrow 0\atop w\to v}{\lim\inf}\langle A^{\circ}(x+tw),w\rangle.
		\end{equation}
		Indeed, by the monotonicity of $A$, for $x^* \in Ax$ we have
		$$\langle x^*, w\rangle = \frac{1}{t} \langle x^*, (x+tw)-x \rangle \le \langle A^{\circ}(x+tw),w\rangle$$
		for all $t \ge 0, w \in X \setminus \{0\}$, 
		with the convention $\langle A^{\circ}(x+tw),w\rangle= +\infty$ if $x+tw\notin {D}(A)$. It follows that
		$$ \langle x^*, v \rangle \le \underset{t \downarrow 0\atop w\to v}{\lim\inf}\langle A^{\circ}(x+tw),w\rangle.$$
		Since this holds for any $x^* \in Ax$, we get inequality \eqref{t.4.2.0}. It remains to show
		\begin{equation}\label{t.4.2.1}
			\sigma_{Ax}(v) \ge  \underset{t \downarrow 0 \atop w \to v}{\lim\inf}\,\langle A^{\circ}(x+tw),w\rangle.
		\end{equation}	
		This obviously holds for  $\sigma_{Ax}(v) = +\infty$. If $\sigma_{Ax}(v) < +\infty$, then  $v \in \ope{dom}(\sigma_{Ax})\subset  \overline{{D}(\partial \sigma_{Ax})}$. We consider two cases of $v$ as follows.
		
		\textit{Case 1:} $v \in D(\partial \sigma_{Ax})$. Pick any $x^* \in \partial \sigma_{Ax}(v).$ Involving \eqref{*} that $\partial \sigma_{Ax}(v) = A(x;v)$, we have $x^*\in A(x;v)$. 
		According to Proposition \ref{T.s4.1}, we can find nets $(t_i)_{i \in I} \subset \mathbb{R}_{+}$, $ (w_i)_{i\in I} \subset X$ and $(z^*_i)_{i \in I} \subset X^*, (x^*_i)_{i \in I} \subset X^*$ with $t_i \downarrow 0,$  $z^*_i \in J_{t_i}(w_i)$, and $ x^*_i \in A(x+t_iw_i)$ such that
		\begin{equation}\label{t.4.2.2}
			z^*_i \to J(v), \,\, \,\,x^*_i \to x^*.
		\end{equation} 
		On one hand, since $x^*_i \to x^*$,  $(x^*_i)_{i\in I}$ is eventually bounded. Hence $(A^{\circ}(x+t_iw_i))_i$ is also eventually bounded due to $\Vert A^{\circ}(x+t_iw_i)\Vert \le \Vert x^*_i\Vert$ for all $i \in I$. 
		By Banach-Alaoglu theorem, passing to the limit $i\to\infty$ 
		along a subnet, one has
		\begin{equation}\label{t.4.2.3}
			A^{\circ}(x+t_iw_i) \overset{\ope{w}^*}{\rightharpoonup} z^*
		\end{equation}
		for some $z^*\in X^*$. Note that $(w_i)_{i}$ is eventually bounded since $ \Vert w_i\Vert \le \Vert z^*_i\Vert + \sqrt{2t_i}, \forall i \in I $  and $\Vert z^*_i\Vert \to \Vert J(v)\Vert= \Vert v\Vert$. This gives $x+t_iw_i \to x$. It follows from the maximal monotonicity of $A$, the eventual boundedness of $(A^{\circ}(x+t_iw_i))_i$ and \eqref{t.4.2.3} that $z^* \in Ax.$ This yields 
		\begin{equation}\label{t.4.2.4}
			\sigma_{Ax}(v) \ge \langle z^*, v\rangle.
		\end{equation}
		On the other hand, using Lemma \ref{L.s4.3}, we may suppose, by passing to a subnet
		if necessary, that $w_i \to v.$
		Owing to the fact that $(A^{\circ}(x+t_iw_i))_{i \in I}$ is eventually bounded and \eqref{t.4.2.3}, we get
		\begin{equation*}
			\langle A^{\circ}(x+t_iw_i), w_i\rangle \to \langle z^*, v\rangle.
		\end{equation*}
		Together with \eqref{t.4.2.4}, we derive
		\begin{equation}\label{t.4.2.5}
			\sigma_{Ax}(v) \ge \underset{i}{\lim}\,\langle A^{\circ}(x+t_iw_i), w_i\rangle \ge \underset{t \to 0 \atop j_{t}(w) \to J(v)}{\lim\inf}\langle A^{\circ}(x+tw),w\rangle.
		\end{equation}	
		It follows from Lemma \ref{L.s4.3} that
		\begin{equation}\label{t.4.2.6}
			\underset{t \to 0 \atop j_{t}(w) \to J(v)}{\lim\inf}\langle A^{\circ}(x+tw),w\rangle \ge \underset{t \to 0 \atop w \to v}{\lim\inf}\langle A^{\circ}(x+tw),w\rangle.
		\end{equation}
		Combining \eqref{t.4.2.5} and \eqref{t.4.2.6}, we obtain \eqref{t.4.2.1}.

		\textit{Case 2:} $ v \in \ope{dom}\sigma_{Ax} \setminus \ope{D}(\partial \sigma_{Ax})$. This means that $\partial \sigma_{Ax}(v) = \varnothing.$ According to \cite[Theorem 2]{BroRoc}, there exists $v_n \to v$ such that
		$\sigma_{Ax}(v_n) \to \sigma_{Ax}(v)$ and $\partial \sigma_{Ax}(v_n) \ne \varnothing.$ From Case 1, for every $n \in \mathbb{N}$
		$$ \sigma_{Ax}(v_n)= \underset{t \downarrow 0, \atop w \to v_n}{\lim\inf}\langle A^{\circ}(x+tw),w\rangle.$$
		Hence, for every $n$ there exist $t_n >0, w_n \in X$ such that
		$$ t_n \le \frac{1}{n},\,\, \Vert w_n -v_n\Vert \le \frac{1}{n}, \,\, \vert \langle A^{\circ}(x+t_nw_n), w_n\rangle -\sigma_{Ax}(v_n)\vert \le \frac{1}{n}.$$
		This implies that $t_n \to 0, w_n \to v$ and $\underset{n \to \infty}{\lim}\langle A^{\circ}(x+t_nw_n), w_n\rangle = \sigma_{Ax}(v)$. Hence we get \eqref{t.4.2.1}. The proof is complete.
	\end{proof}

	The next result follows directly from Lemma \ref{l.m.t} and the previous theorem.
	
	\begin{corollary}\label{C.s4.1}
		Let $X$ be a Banach space such that $X, X^*$ are locally uniformly convex and $X^{**}$ is strictly convex and  let $A: X \rightrightarrows X^*$ be a maximal monotone operator of type (D).  
		Then for $x \in {D}(A)$ and $v \in X\setminus\{0\}$
		\begin{equation}\label{C.s4.2}
			\sigma_{Ax}(v)= 
			\underset{t \downarrow 0, \atop w \to v}{\lim\inf}\,\langle A^{\circ}(x+tw),w\rangle.
		\end{equation}
	\end{corollary}
	
	\begin{corollary} Let $X$ be a Banach space such that $X^*, X^{**}$ are strictly convex and $X$ has the w-Kadec-Klee property, $X^*$ has the $w^*$-Kadec-Klee property. Let $A_1, A_2: X \rightrightarrows X^*$ be maximal monotone operators of type (D). If ${D}(A_1)= {D}(A_2)=: D$ and $A^{\circ}_1x \in A_2x$ for all  $x \in D,$
		then $A_1 = A_2.$	
	\end{corollary}
	\begin{proof}
		We first show that for $x \in D$ and $v \in X \setminus \{0\}$
		\begin{equation}\label{c.4.2.1}
			\sigma_{A_1x}(v) \ge \sigma_{A_2x}(v),
		\end{equation}
		which would imply $A_2x \subset A_1x$ for $x \in {D}$. This is trivial for $\sigma_{A_1x} = +\infty$. We consider $\sigma_{A_1x}(v) < +\infty$. According to Theorem \ref{T.s4.2}, there exist nets $(t_i)_{i \in I} \subset \mathbb{R}_{+}, (w_i)_{i \in I} \in X$ with $t_i \to 0$ and $w_i \to v$ such that
		\begin{equation}\label{c.4.2.2}
			\sigma_{A_1x}(v) = \underset{i}{\lim}\langle A^{\circ}_1(x+t_iw_i),w_i\rangle.
		\end{equation}
		By assumption, $A^{\circ}_1(x+t_iw_i) \subset A_2(x+t_iw_i)$ for all $i \in I$. The monotonicity of $A_2$ implies that
		$$ \sigma_{A_2x}(w_i) \le \langle A^{\circ}_1(x+t_iw_i), w_i\rangle \ \ \mbox{for all } i \in I.$$
		Combining this with the lower semicontinuity of $\sigma_{A_2x}$ and \eqref{c.4.2.2}, we obtain
		$$ \sigma_{A_2x}(v) \le \underset{i}{\lim}\,\sigma_{A_2x}(w_i) \le \underset{i}{\lim}\langle A^{\circ}_1(x+t_iw_i), w_i\rangle \le \sigma_{A_1x}(v),$$
		which means that \eqref{c.4.2.1} holds. This yields $A^{\circ}_2x \in A_2x \subset A_1x$ for all $x \in D.$ Similarly, we have  $A_1x \subset A_2x$ for every $x \in D$. Therefore $A_1= A_2$, as was to be shown.\end{proof}
	
	\noindent{\bf Acknowledgements:}  A part of this paper was done during the first author's stay at Vietnam Institute for Advanced Study in Mathematics (VIASM). He would like to thank
	VIASM for its wonderful working condition. This work was supported by Vietnam National Program for
	the Development of Mathematics until 2030 (project code: B2022-CTT-07).
	
\end{document}